\documentclass[10pt]{amsart}
\usepackage{amsmath}
\usepackage{amsthm}
\usepackage{amsfonts}
\usepackage{amssymb}
\usepackage[dvips]{graphicx}

\input{epsf}

\allowdisplaybreaks
\newtheorem{theorem}{Theorem}[section]
\newtheorem{lemma}[theorem]{Lemma}

\newtheorem{corollary}[theorem]{Corollary}
\theoremstyle{definition}
\newtheorem{definition}[theorem]{Definition}
\newtheorem{remark}[theorem]{Remark}
\newtheorem{example}[theorem]{Example}
\newtheorem{question}{Question}
\newcommand{\p}{{\mathfrak p}}
\newcommand{\g}{{\mathfrak g}}
\newcommand{\x}{\underline x}
\newcommand{\y}{\underline y}

\newcommand{\ind}{\operatorname{ind}}
\newcommand{\spn}{\operatorname{span}}
\begin{document}
\frenchspacing

%  Headings
%\renewcommand{\evenhead}{V.  Coll, A. Giaquinto and C. Magnant}
%\renewcommand{\oddhead}{Meanders and Frobenius seaweed algebras}

%  Titlepage

%\thispagestyle{empty}

% Editorial
%\FirstPageHead{*}{*}{20**}{\pageref{firstpage}--\pageref{lastpage}}
%  Parameters: Volume, number, year, page range, paper type
% End Editorial

%\label{firstpage}
% Note for Author!
% Insert  \label{lastpage} at the end of the file

\title{Meander graphs and Frobenius Seaweed Lie algebras}
% author one information
\author{Vincent Coll}
\address{Department of Mathematics\\College of Arts and Sciences\\Lehigh University\\Bethlehem, PA 18015}
\email{vec208@lehigh.edu}

\author{Anthony Giaquinto}
\address{Department of Mathematics and Statistics\\College of Arts and Sciences\\
Loyola University Chi\-cago\\
Chicago, Illinois 60626 USA}
%\curraddr{}
%\thanks{}
\email{tonyg@math.luc.edu}

\author{Colton Magnant}
\address{Department of Mathematics\\College of Arts and Sciences\\Lehigh University\\Bethlehem, PA 18015}
\email{dr.colton.magnant@gmail.com}

%\begin{document}
\maketitle

\begin{abstract}
The index of a seaweed Lie algebra can be computed from its associated
meander graph. We examine this graph in several ways with a goal of
determining families of Frobenius (index zero) seaweed algebras. Our
analysis gives two new families of Frobenius seaweed algebras as well
as elementary proofs of known families of such Lie algebras.
\end{abstract}

\section{Introduction}\noindent
Let $\mathfrak L$ be a Lie algebra over a field of characteristic
zero. For any functional $F\in \mathfrak L^*$ there is an associated
skew bilinear form $B_F$ on $\mathfrak L$ defined by $B_F(x,y)=F([x,y])$ for
$x,y\in \mathfrak L$. The index of $\mathfrak L$ is defined to be
$$\ind \mathfrak L = \min_{F\in \mathfrak L^*}\dim (\ker (B_F)).$$ The
Lie algebra $\mathfrak L$ is {\it{Frobenius}} if $\dim \mathfrak =0$;
equivalently, if there is a functional $F\in \mathfrak L^*$ such
that $B_F(-,-)$ is non-degenerate.

Frobenius Lie algebras were first studied by Ooms in \cite{Ooms} where
he proved that the universal enveloping algebra $U\mathfrak L$ is
primitive (i.e.\ admits a faithful simple module) provided that
$\mathfrak L$ is Frobenius and that the converse holds when $\mathfrak
L$ is algebraic. The relevance of Frobenius Lie algebras to
deformation and quantum group theory stems from their relation to the
classical Yang-Baxter equation (CYBE). Suppose $B_F(-,-)$ is
non-degenerate and let $M$ be the matrix of $B_F(-.-)$ relative to
some basis $\{x_1,\ldots , x_n\}$ of $\mathfrak L$. Belavin and
Drinfel'd showed that $r=\sum_{i,j}(M^{-1})_{ij}x_i\wedge x_j$ is a
(constant) solution of the CYBE, see \cite{BD}. Thus, each pair
consisting of a Lie algebra $\mathfrak L$ together with functional
$F\in \mathfrak L^*$ such that $B_F$ is non-degenerate provides a
solution to the CYBE, see \cite{GG:Boundary} and \cite{GG:Frob} for
examples.

The index of a semisimple Lie algebra $\g$ is equal to its rank and
thus such algebras can never be Frobenius. However, there always exist
subalgebras of $\g$ which are Frobenius. In particular, many amongst
the class of {\it biparabolic} subalgebras of $\g$ are Frobenius. A
biparabolic subalgebra is the intersection of two parabolic
subalgebras whose sum is $\g$. They were first introduced in the case
$\g=\mathfrak{sl}(n)$ by Dergachev and Kirillov in \cite{DK} where
they were called Lie algebras of {\it seaweed} type. Associated to
each seaweed algebra is a certain graph called the {\it meander}. One
of the main results of \cite{DK} is that the algebra's index is
determined by graph-theoretical properties of its meander, see Section
\ref{sec:meander} for details.

Using different methods, Panyushev developed an inductive procedure
for computing the index of seaweed subalgebras, see
\cite{Panyushev}. In the same paper, he exhibits a closed form for the
index of a biparabolic subalgebra of $\mathfrak{sp}(n)$.  One may also see \cite{Khoroshkin, Stolin}.

Tauvel and Yu found in \cite{Tauvel-Yu} an upper bound for the index
of a biparabolic subalgebra of an arbitrary semisimple Lie algebra,
and they conjectured that this was an equality. Joseph proved the
Tauvel-Yu conjecture in \cite{Joseph}.

The methods of \cite{DK}, \cite{Panyushev}, \cite{Tauvel-Yu}, and
\cite{Joseph} are all combinatorial in nature. Yet even with the this
theory available, it is difficult in practice to
implement this theory to find families of Frobenius biparabolic Lie
algebras. In contrast, for many cases it is known explicitly which
biparabolic algebras have the maximum possible index. For example, the
only biparabolics in $\mathfrak{sl}(n)$ and $\mathfrak{sp}(n)$ which
have maximal index are the Levi subalgebras. In contrast, the problem
of determining the biparabolics of minimal index is an open question
in all cases.

Our focus in this note is on the seaweed Lie algebras --
these are the biparabolic subalgebras of $\mathfrak{sl}(n)$. The only
known families of Frobenius seaweed Lie algebras that seem to be in the
literature will be outlined in Section \ref{sec:families}, although
the unpublished preprint \cite{Elashvili2} may offer more examples.
We shall examine these families using the meander graphs of Dergachev
and Kirillov. Our methodology provides new proofs that these algebras
are indeed Frobenius. We also exhibit a new infinite family of
Frobenius seaweed Lie algebras in Section~\ref{sec:submaximal}.

\section{Seaweed Lie algebras}

In this section we introduce the seaweed Lie algebras of
\cite{DK}. Recall that a composition of a positive integer $n$ is an
unordered partition $\underline x=(a_1,\ldots, a_m)$. That is,
each $a_i\geq 0$ and $\sum a_i = n$.

\begin{definition} Let $V$ be an $n$-dimensional vector space with a basis
$e_1,\ldots, e_n$. Let $\underline x=(a_1,\ldots, a_m)$ and $\underline y =(b_1,\ldots, b_t)$ be two compositions of $n$ and consider the flags
$$\{0\}\subset V_1\subset \cdots \subset V_{m-1}\subset V_m =V\quad
  \mbox{and}\quad V=W_0\supset W_1\supset \cdots \supset W_t=\{0\}$$
  where $V_i=\spn\{e_1,\ldots, e_{a_1+\cdots +a_i}\}$ and
  $W_j=\spn\{e_{{b_1}+\cdots +e_{b_j+1}}, \ldots e_n\}$. The
  subalgebra of $\mathfrak{sl}(n)$ preserving these flags is called a
  seaweed Lie algebra and is denoted $\p(\underline x \mid \underline
  y)$.
\end{definition}

A basis-free definition is available but is not necessary for the
present discussion. The name seaweed Lie algebra was chosen due to
their suggestive shape when exhibited in matrix form. For example, the
algebra $\p(3,1,3,2\,\,|\,\,4,2,3)$ consists of traceless matrices of the form

$$\begin{bmatrix}
*&*&*&*&\cdot&\cdot&\cdot&\cdot&\cdot\\
*&*&*&*&\cdot&\cdot&\cdot&\cdot&\cdot\\
*&*&*&*&\cdot&\cdot&\cdot&\cdot\\
\cdot&\cdot&\cdot&*&*&*&\cdot&\cdot&\cdot\\
\cdot&\cdot&\cdot&\cdot&*&*&\cdot&\cdot&\cdot\\
\cdot&\cdot&\cdot&\cdot&*&*&\cdot&\cdot&\cdot\\
\cdot&\cdot&\cdot&\cdot&*&*&*&*&*\\
\cdot&\cdot&\cdot&\cdot&\cdot&\cdot&*&*&*\\
\cdot&\cdot&\cdot&\cdot&\cdot&\cdot&*&*&*\\
\end{bmatrix}$$
where the entries marked by the dots are zero.

Many important subalgebras of $\mathfrak{sl}(n)$ are of
seaweed type, as illustrated in the following example.
\begin{example}\label{seaweed-examples}
\begin{itemize}
\item The entire algebra $\mathfrak{sl}(n)=\p(n\,\,|\,\,n)$ has index $n-1$.
\item The Cartan subalgebra of traceless diagonal matrices is
  $\p(\underline 1\,\,|\,\, \underline 1)$, where $\underline 1 =
  (1,1,\ldots, 1)$ and has index $n-1$.
\item The Borel subalgebra is $\p(\underline 1\,\, |\,\, n)$
  and has index $\lfloor (n+1)/2\rfloor$.
\item A maximal parabolic subalgebra is of the form
  $\p(a,b\,\,|\,\,n)$. Elashvili proved in \cite{Elashvili} that
  its index is $\gcd (a,n) -1$.
\end{itemize}
\end{example}

The only explicitly known Frobenius examples in the above list are the maximal
parabolic algebras $\p(a,b\,\,|\,\,n)$ with $a$ and $n$ relatively
prime. Of course, 
another infinite family of Frobenius seaweed algebras occurs
when $\underline a = (2,\ldots, 2,1)$, $\underline b =(1,2,\ldots 2)$,
and $n$ is odd. A similar case is $\underline a=(1,2\ldots, 2,1)$,
$\underline b =(2,\ldots,2)$, and $n$ is even. These two families are
detailed in \cite{Panyushev}.

A tantalizing question is how
to classify which seaweed algebras are Frobenius, especially given
their importance in the general theory of Lie algebras and
applications to deformations and quantum groups.

\section{Meanders}\label{sec:meander}

As stated earlier, Dergachev and Kirillov have developed a
combinatorial algorithm to compute the index of an arbitrary
$\p(\underline x\,\, |\,\, \underline y)$ from its associated {\it
  meander} graph $M(\underline x\,\,|\,\,\underline y)$ determined by
the compositions $\underline x$ and $\underline y$. The vertices of
$M(\underline x \,\,|\,\, \underline y)$ consist of $n$ ordered points
on a horizontal line, which can be called $1,2,\ldots,n$. The edges
are arcs above and below the line connecting pairs of different
vertices.

More specifically, the composition $\underline x =
(a_1,\ldots, a_m)$ determines arcs above the line which we will call
the top edges. The component $a_1$ of $\underline x$ determines
$\lfloor a_1/2\rfloor$ arcs above vertices $1,\ldots, a_1$. The arcs
are obtained by connecting vertex 1 to vertex $a_1$, vertex 2 to
vertex $a_{1}-1$, and so on. If $a_1$ is odd then vertex $a_{\lceil
  a_1/2\rceil}$ has no arc above it. For the component $a_2$ of
$\underline a$, we do the same procedure over vertices $a_1+1,\ldots,
a_1+a_2$, and continue with the higher $a_i$.

The arcs corresponding to $\underline y=(b_1,\ldots,b_t)$ are drawn
with the same rule but are under the line containing the
vertices. These are called the bottom edges.

It is easy to see that every meander consists of
a disjoint union of cycles, paths, and isolated points, but not all of these are necessarily present in any given meander.

\begin{theorem}[Dergachev-Kirillov] The index of the Lie algebra of seaweed type $\p(\underline a\,\,|\,\,\underline b)$ is equal to the number of connected components in the meander plus the number of closed cycles minus 1.
\end{theorem}

\begin{remark} The presence of the minus one in the theorem is due to our use of seaweed subalgebras of $\mathfrak{sl}(n)$ rather than of $\mathfrak{gl}(n)$ as used by Dergachev and Kirillov \cite{DK}. The index drops by one by the
restriction to $\mathfrak{sl}(n)$ from $\mathfrak{gl}(n)$.
\end{remark}

\begin{example}\label{modified}  Figure~\ref{meanderfig1} shows the meander $M(\underline x\,\,|\,\,\underline y)$ corresponding to the compositions
$\underline x = (5,2,2)$ and $\underline y = (2,4,3)$.

\begin{figure}[ht!]
\begin{center}
 \epsfbox{figs.1}
\caption{$M(5, 2, 2\,\,|\,\,2, 4, 3)$. \label{meanderfig1}}
\end{center}
\end{figure}

%\begin{center}
%\setlength{\unitlength}{1cm}
%\begin{picture}(9,2)
%\multiput(0,0)(1,0){9}{\circle*{0.1}}
%\put(2,0){\oval(4,1.5)[t]}
%\put(2,0){\oval(2,1)[t]}
%\put(5.5,0){\oval(1,1)[t]}
%\put(7.5,0){\oval(1,1)[t]}
%\put(0.5,0){\oval(1,1)[b]}
%\put(3.5,0){\oval(3,1.5)[b]}
%\put(3.5,0){\oval(1,1)[b]}
%\put(7,0){\oval(2,1)[b]}
%\end{picture}
%\end{center}

\bigskip

We see that there is a single path and a single cycle. Using the
theorem above, the index is $2+1-1=2$. Hence, $\mathfrak{p}(5, 2, 2\,\, |\,\, 2, 4, 3)$ is not a Frobenius
algebra.

\end{example}

It is easy to see that to obtain a Frobenius algebra, the only
possibility for the meander is that it consist of a single path with
no cycles and no isolated points. The following illustrates this
point.

\begin{example}\label{examplea}
Consider the algebra
$\p(3,2,2\,\, | \,\, 2,5)$. Its meander is given in Figure~\ref{meanderfig2}.

\begin{figure}[ht!]
\begin{center}
 \epsfbox{figs.2}
 \caption{$M(3, 2, 2\,\,|\,\,2, 5)$. \label{meanderfig2}}
\end{center}
\end{figure}

%\begin{center}
%\setlength{\unitlength}{1cm}
%\begin{picture}(9,2)
%\multiput(0,0)(1,0){7}{\circle*{0.1}}
%\put(1,0){\oval(2,1)[t]}
%\put(3.5,0){\oval(1,1)[t]}
%\put(5.5,0){\oval(1,1)[t]}
%\put(0.5,0){\oval(1,1)[b]}
%\put(4,0){\oval(4,1.5)[b]}
%\put(4,0){\oval(2,1)[b]}
%\end{picture}
%\end{center}

\bigskip

Labeling the vertices with $\{1, 2, \dots, n\}$ from left to right, notice that $M(3,2,2\,\,|\,\,2,5)$  is the single path $2,1,3,7,6,4,5$ (if we
start with $2$) or its reversal $5,4,6,7,3,1,2$ if we start with $5$. In particular, the index is $1-1=0$ and so this is a Frobenius algebra.
\end{example}

\begin{question} What are the conditions on the compositions $\underline x$ and $\underline y$ so that the meander $M(\underline x\,\,|\,\,\underline y)$ consists of a single path with no cycles or isolated points?
\end{question}

As stated, this seems to be an elementary question involving nothing
more that the basics of graph theory.  However, the apparent
simplicity of the question is misleading since an answer would provide
a complete classification of Frobenius seaweed algebras - a difficult
problem.  Even so, it is easy to give some necessary conditions on
$\underline x =(a_1,\ldots ,a_m)$ and $\underline y = (b_1, \ldots,
b_t)$ for $M(\underline x\,\,|\,\, \underline y)$ to be a single
path. For example, exactly two elements of the set $\{a_1,\ldots,
a_m,b_1,\ldots,b_t)$ must be odd. This is because a path must have a
starting and ending point, and these corresponds to vertices of degree
one. A vertex of degree one is either missing a top edge or bottom
edge connecting to it, and this happens only if some $a_i$ or $b_j$ is
odd.

Another necessary condition for $M(\underline x\,\,|\,\, \underline
y)$ to be a single path is that $a_1\neq b_1$. In this
case, $$\p(\underline x \,\, | \,\,\underline y)\simeq
\mathfrak{sl}(a_1)\bigoplus \p(a_2,\ldots,
a_m\,\,|\,\,b_2,\ldots,b_t)$$ and thus $\p(\underline x \,\, |
\,\,\underline y)$ is not Frobenius since the index is additive for
direct sums of Lie algebras. More generally, if $\sum_{i=1}^ra_i
=\sum_{j=1}^rb_j$ for some $r\leq \min \{m,t\}$ then the meander is
not a single path.  Other necessary conditions can be given, but none
seems to shed light on what is sufficient.

\section{Families of Frobenius seaweed algebras}\label{sec:families}

In this section we revisit some known families of Frobenius seaweed
algebras in terms of meanders. At the end we also provide two new families.

 First consider Panyushev's example with $\underline x = (2,\ldots,
 2,1)$, $\underline y =(1,2,\ldots 2)$, and $n$ is odd. Again,
 numbering as in Example~\ref{examplea}, the top edges connect 2 to 4,
 4 to 6, etc.\ while the bottom edges connect 1 to 3, 3 to 5,
 etc. Hence, the meander consists of the single path $1,2,\ldots,
 n$. A similar argument verifies that the meander for $\underline
 x=(1,2\ldots, 2,1)$ and $\underline y =(2,\ldots,2)$ with $n$ even is
 also the path $1,2,\ldots, n$.

To analyze some other cases it is convenient to modify the definition
of the meander $M(\underline x \,\,|\,\,\underline y)$.

\begin{definition} Suppose $\x$ and $\y$ are compositions of $n$. The modified meander $M'(\x \,\,| \,\,\y)$ is the graph $M(\x\,\,|\,\,\y)$ appended with a loop corresponding to each odd $a_i$ and $b_j$.  Specifically, for all odd $a_i$, add
a loop connecting $a_1+\cdots +a_{i-1}+\lceil a_i/2\rceil $ to
itself. Similarly, for all odd $b_j$, add a bottom loop connecting
$b_1+\cdots +b_{j-1}+\lceil b_j/2\rceil$ to itself.
\end{definition}

Note that in $M'(\underline x \,\, | \,\, \underline y)$ each vertex
is incident with exactly one top and one bottom edge or loop.

\begin{example}\label{modified2}
Below is the modified meander $M'(5,2,2\,\, |\,\, 2,4,3)$. Compare
with the meander $M(5,2,2\,\,|\,\,2,4,3)$ given in Example \ref{modified}.

\begin{figure}[ht!]
\begin{center}
 \epsfbox{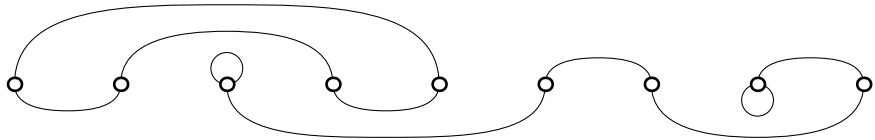}
 \caption{$M(5, 2, 2\,\,|\,\,2, 4, 3)$ with loops. \label{meanderfig3}}
\end{center}
\end{figure}

%\begin{center}
%\setlength{\unitlength}{1cm}
%\begin{picture}(9,2)
%\multiput(0,0)(1,0){9}{\circle*{0.1}}
%\put(2,0){\oval(4,1.5)[t]}
%\put(2,0){\oval(2,1)[t]}
%\put(2,.2){\circle{.3}}
%\put(5.5,0){\oval(1,1)[t]}
%\put(7.5,0){\oval(1,1)[t]}
%\put(0.5,0){\oval(1,1)[b]}
%\put(3.5,0){\oval(3,1.5)[b]}
%\put(3.5,0){\oval(1,1)[b]}
%\put(7,0){\oval(2,1)[b]}
%\put(7,-.2){\circle{.3}}
%\end{picture}
%\end{center}

\bigskip

\end{example}

\subsection{The top and bottom bijections} Each modified meander determines two bijections of $S=\{1,2,\ldots, n\}$ to itself. Define a
``top'' bijection $t$ of $S$ by $t(i)=i$, where $j$ is the unique
vertex incident with the same top edge as $i$. If $i$ is joined to
itself by a top loop, then $t(i)=1$.  In a similar way, define a
``bottom'' bijection $b$ of $S$ by $b(i)=j$, where $j$ is the unique vertex
incident with the same bottom edge as $j$. If $i$ is joined to itself
by a bottom loop, then $b(i)=1$.  Clearly the maps $t$ and $b$ are
well-defined. For instance, in Example \ref{modified2}, we have
$t(3)=3$ and $b(3)=6$.

\begin{definition}
Let $\x$ and $\y$ be compositions of $n$. The meander permutation
$\sigma_{\x,\y}\in S_n$ is the permutation $t\circ b$ of $S$. That is,
$\sigma_{\x,\y}(i)=t(b(i))$.
\end{definition}

\begin{example} Consider the meander permutation $\sigma_{\x,\y}$ with $\x$ and
$\y$ as in Example~\ref{modified2}.  We can write $\sigma_{\x,\y}$ as a product of disjoint cycles in $S_{n}$: $(1,4)(2,5)(3,7,8,9,6)$ (note the different use of the term ``cycle'').
\end{example}

\begin{theorem}\label{cycle} Suppose $\underline x$ and $\underline y$ are compositions of $n$.  Then the meander $M(\underline x\,\,|\,\, \underline y)$ is a single path if and only if the meander permutation $\sigma_{\x,\y}$ is an $n$-cycle in $S_n$.
\end{theorem}

\begin{proof} Suppose the meander $M(\underline x\,\,|\,\, \underline y)$ is the single path $a_1, a_2, \ldots, a_n$. By switching $\x$ and $\y$ if necessary, we can assume that $b(a_1)=a_2$.
Then the meander permutation is the $n$-cycle $(a_1, a_3, \ldots
a_{n-1}, a_n, a_{n-2},\ldots a_2)$ if $n$ is even and if $n$ is odd it
is the $n$-cycle if $(a_1, a_3, \ldots, a_n, a_{n-1}, a_{n-3},\ldots,
a_2).$

Conversely suppose $\sigma_{\x,\y}$ is an $n$-cycle but
$M(\x\,\,|\,\,\y)$ is not a single path. Then $M(\x\,\,|\,\,\y)$
contains either an isolated point, a path of length less than $n$, or
a cycle. We shall show that each of these possibilities leads to a
contradiction.

If $i$ is an isolated point of $M(\x\,\,|\,\, \y)$, then it is a fixed
point of $\sigma_{\x,\y}$ which therefore can not be an $n$-cycle.

If $a_1, \ldots a_k$ is a path in $M(\x\,\,|\,\,\y)$ with $k<n$ then,
depending on whether $k$ is even or odd, either the $(a_1, a_3, \ldots
a_{k-1}, a_k, a_{k-2},\ldots a_2)$ or $(a_1, a_3, \ldots, a_k,
a_{k-1}, a_{k-3},\ldots, a_2)$ appears in the cycle decomposition of
$\sigma_{\x,\y}$. Since $k<n$ we conclude that $\sigma_{\x,\y}$ is not
an $n$-cycle.

Now if $M(\x\,\,|\,\,\y)$ contains a cycle $a_1,a_2,\ldots, a_k,a_1,$
then the meander permutation contains either the $k/2$ cycle
$(a_1,a_3,\ldots,a_{n-1})$ if $n$ is even or the $k$-cycle $(a_1, a_3,
\ldots, a_n, a_2, a_4,\ldots, a_{n-1})$ if $n$ is odd. If $k<n$ then
$\sigma_{\x,\y}$ is not an $n$-cycle. If $k=n$ is even, then the same
argument shows that $\sigma_{\x,\y}$ is not an $n$-cycle. The
remaining case is that $k=n$ is odd. If this happens though, we must
have $M(\x\,\,|\,\,\y)=M'(\x\,\,|\,\,\y)$, and consequently all
components $a_i$ and $b_j$ are even. Since $\sum a_i=n$ we have a
contradiction. Thus, in all cases when $M(\x\,\,|\,\,\y)$ is not a
single path, the meander permutation $\sigma_{\x,\y}$ is not an
$n$-cycle, which is a contradiction. The proof is complete.
\end{proof}

\subsection{Maximal Parabolic Subalgebras} To generate more examples of Frobenius Lie algebras, we consider maximal parabolic seaweed subalgebras of $\mathfrak{sl}(n)$ which are necessarily of the form $\p(a,b\,\,|\,\,n)$.

\begin{lemma}\label{modn} Consider the compositions $\underline x=(a,b)$ and $\underline y = n$. The meander permutation $\sigma_{\x,\y}$ is the map
sending $i$ to $i+a \mod n$ for all $i$.
\end{lemma}
\begin{proof}
By definition of the top and bottom maps, we have
$$b(i)=n+1-i \quad \mbox{and}\quad t(i)=\begin{cases}{a+1}&\mbox{if}\quad 1\leq i\leq a\\
n+a+1-i& \mbox{if} \quad a+1\leq i\leq n\end{cases}$$
and thus
$$t(b(i))=\begin{cases}a-n+i& \mbox{if}\quad 1\leq b(i)\leq a\\ a+i&
\mbox{if}\quad a+1\leq b(i)\leq n\end{cases}.$$ Therefore
$\sigma_{\x,\y}(i)=t(b(i))=i+a\mod n$.
\end{proof}

Recall Elashvili's result asserting that the maximal parabolic algebra
$\p(a,b\,\,|n)$
is Frobenius if and only if $\gcd(a,n)=1$. An immediate corollary of
the previous lemma gives a new simple proof of Elashvili's result.

\begin{corollary}\label{relprime}
The maximal parabolic algebra $\p(a,b\,\,|\,\,n)$ is Frobenius if and
only if $\gcd(a,n)=1$.
\end{corollary}
\begin{proof}

By Theorem \ref{cycle} it suffices to show that the meander
permutation is an $n$-cycle. According to Lemma \ref{modn},
$\sigma_{\x,\y}(i)= i+a\mod n$ for all $i$. Thus, the meander
permutation is an $n$-cycle if and only if the sequence $i, i+a, i+2a, \ldots,
i+(n-1)a$ forms a complete residue system modulo $n$. This occurs
precisely when $\gcd(a,n)=1$. The proof is complete.
\end{proof}

\subsection{Opposite maximal parabolic subalgebras}\label{sec:opposite}
We now use the same ideas to present another family of Frobenius seaweed
algebras each of which is an intersection of a positive and negative
maximal parabolic algebra. Such algebras are of the form $\p(a,b
\,\,|\,\, c,d)$ and are called {\it opposite maximal parabolic
  subalgebras}.

\begin{lemma}Let $\x=(a,b)$ and $\y=(c,d)$ be compositions of $n$. The permutation meander $\sigma_{\x,\y}$ is the map sending $i$ to $a-c\mod n$ for all $i$.
\end{lemma}

\begin{proof}
The bottom and top maps are given by
$$b(i)=\begin{cases}c+1-i & \mbox{if}\quad 1\leq i\leq c\\ n+c+1-i&
\mbox{if}\quad c+1\leq i\leq n\end{cases},$$
and
$$t(i)=\begin{cases}a+1-i & \mbox{if}\quad 1\leq i\leq a\\ n+a+1-i&
\mbox{if}\quad a+1\leq i\leq n\end{cases}.$$ There are four possible
compositions $t(b(i))$, depending on and whether $i\leq c$ or $i>c$ and
whether $b(i)\leq a$ or $b(i)>a$. It is an easy calculation to see
that in each case $t(b(i))=a-c+i\mod n$.
\end{proof}

An immediate consequence is the following result.

\begin{corollary}
The opposite maximal parabolic seaweed algebra $\p(a,b\,\,|\,\,c,d)$
is Frobenius if and only if $\gcd(a-c,n)=1$.
\end{corollary}
\begin{proof}
The argument is exactly as that used in
Corollary~\ref{relprime}. Namely, that the meander permutation is an
$n$-cycle if and only if the sequence $i, i+(a-c), i+2(a-c), \ldots ,
i+(n-1)(a-c)$ is a complete residue system modulo $n$, and this is the
case if and only if $\gcd(a-c,n)=1$.
\end{proof}

The result of the corollary was first proved using different methods by Stolin in \cite{Stolin}. For example, the Lie algebra $\p(2,3\,\,|\,\,4,1)$ is
Frobenius since $2-4=-2$ is relatively prime to $7$.

At this time, the above line of reasoning does not easily extend to
compositions $\x$ and $\y$ with more than two components. However,
some calculations offer hope of producing more families of
Frobenius Lie algebras using methods similar to those above.

\subsection{Submaximal parabolic algebras}\label{sec:submaximal}
We conclude with another new family of Frobenius algebras. These are
of the form $\p(a,b,c\,\,|\,\,n)$, so they are parabolic algebras
omitting exactly two simple roots. We use a different technique than
for maximal or opposite maximal algebras to analyze this family. Our
result is the following classification theorem.

\begin{theorem}\label{submaximal} The submaximal parabolic algebra $\p(a,b,c\,\,|\,\,n)$ is Frobenius if and only if $\gcd(a+b,b+c)=1$.
\end{theorem}

We first establish some conditions on the degrees
of the vertices $\{v_{1}, v_{2}, \dots, v_{n}\}$ of the meander $M=M(a,b,c\,\,|\,\,n)$. Since the
vertices of $M$ are viewed as the numbers $\{1,2,\ldots,n\}$ on a
line the \emph{interval} between vertices $v_{i}$ and $v_{i + 1}$
makes sense.

\begin{lemma}\label{twoones} Suppose $\gcd(a+b,b+c)=1$. Then
there are exactly two vertices of degree $1$ in $M$ and all other vertices have degree $2$.
\end{lemma}

\begin{proof} Suppose for a moment that there exists a
vertex $v$ of degree $0$.  This vertex must have no bottom edge,
meaning that $n$ is odd and $v = v_{(n + 1)/2}$.  We also know $v$ has
no top edge so $b$ is odd and $v$ is halfway between $v_{a + 1}$ and
$v_{a + b}$.  This implies that $a = c$ so $a + b = b + c$, a
contradiction.  Hence, we get exactly one vertex of degree $1$ for
each integer in $\{ a, b, c, n\}$ which is odd.

If $n$ is odd, the vertex $v_{(n + 1)/2}$ has degree $1$.  If all
three of $a$, $b$ and $c$ are odd, then $a + b$ and $b + c$ are both
even, meaning they have a common factor of $2$, a contradiction.  This
implies that exactly one of $a$, $b$ or $c$ must be odd.  Then there
is exactly one other vertex of degree one as desired.

If $n$ is even, the bottom edges form a perfect matching.  If all
three of $a$, $b$ and $c$ are even, then $a + b$ and $b + c$ are again
even, a contradiction.  This implies that exactly two of $a$, $b$ or
$c$ are odd, meaning there are two vertices of degree $1$ as
desired.
\end{proof}

By Lemma \ref{twoones}, one component of $M$ must be a path and there
are possibly more components which are all cycles.  Let $P$ be this
path and suppose $P$ has $a' \leq a$ vertices in the first part of the
partition, and $b' \leq b$ and $c' \leq c$ vertices in the other parts
respectively.  Note that one of $a', b'$ or $c'$ may be zero.  Label
the vertices of $P$ with $u_{1}, u_{2}, \dots, u_{n'}$ where $n' =
|P|$ following the inherited order (the order of the labels $v_{i}$)
of the vertices.  Notice that the path $P$ forms a meander graph on
its own.  This means that, by the proof of Lemma \ref{twoones}, we know that exactly two of the integers in
$\{a', b', c', n'\}$ are even and two are odd.

Now suppose there exists at least one component of $M$ that is a
cycle.  Let $C$ be the set of all vertices in cycles of $M$.  Suppose
$C$ has $d$ vertices in the interval between $u_{i}$ and $u_{i + 1}$.
For the moment, let us suppose that $i \neq \frac{n'}{2}$.  Following
the bottom edges, this means that $C$ must also have $d$ vertices in
the interval between $u_{n' - i}$ and $u_{n' - i + 1}$.  Using this
argument, we will show that $C$ has $d$ vertices in almost every
interval.

Define a \emph{dead end} in $M$ to be an interval $u_{i}$ to $u_{i + 1}$ such
that $M$ contains an edge joining $u_i$ and $u_{i+1}$. In particular,
if $n'$ is even, then the interval between $u_{n'/2}$ and $u_{n'/2 + 1}$
is a dead end.

\begin{lemma}\label{twodeadends}
Suppose $\gcd(a+b,b+c)=1$. Then there are exactly two dead ends in $M$.
\end{lemma}

\begin{proof}
A dead end is formed by two consecutive vertices of $P$ which are
adjacent. Each occurrence of a dead end coincides with one of $a',b',c'$
or $n'$ being even, and we know that exactly two of these are
even. Thus, there are exactly two dead ends and the proof is complete.
\end{proof}

\begin{proof}[Proof of Theorem \ref{submaximal}]

Call an interval a \emph{partition interval} if it is the meeting
point of two parts of our partition.  Namely, the partition intervals
are from $u_{a'}$ to $u_{a' + 1}$ and from $u_{a' + b'}$ to $u_{a' +
  b' + 1}$.  Now suppose $C$ has $d$ vertices in the interval from
$u_{i}$ to $u_{i + 1}$.  For the moment we suppose this interval is
not a dead end.  As mentioned before, this means that, by following
bottom edges, $C$ must also have $d$ vertices in the interval from
$u_{n' - i}$ to $u_{n' - i + 1}$.  Also, by following top edges, $C$
must have $d$ vertices in another interval (depending where the top
edges go).

If the interval from $u_{i}$ to $u_{i + 1}$ happens to be one of the
two partition intervals (for example suppose $i = a'$) then this means
$C$ must have $d_{1}$ vertices in the interval outside $u_{1}$ and at
least $d_{2}$ vertices in the interval from $u_{a' + b'}$ to $u_{a' +
  b' + 1}$ where $d_{1} + d_{2} = d$.  This then implies that $C$ has
$d_{1}$ vertices in the interval beyond $u_{n}$ (following bottom
edges) and another $d_{1}$ vertices in the interval from $u_{a' + b'}$
to $u_{a' + b' + 1}$ (following top edges) for a total of $d$ vertices
in the interval $u_{a' + b'}$ to $u_{a' + b' + 1}$. See
Figure~\ref{meanderfig} for an example.  In this figure, the dark
lines represent the edges of $P$ while light lines represent edges of
$C$.  The unlabeled light lines represent $d$ edges each.  Here $n' =
7$, $a' = 2$, $b' = 2$ and $c' = 3$.

\begin{figure}[ht!]
\begin{center}
 \epsfbox{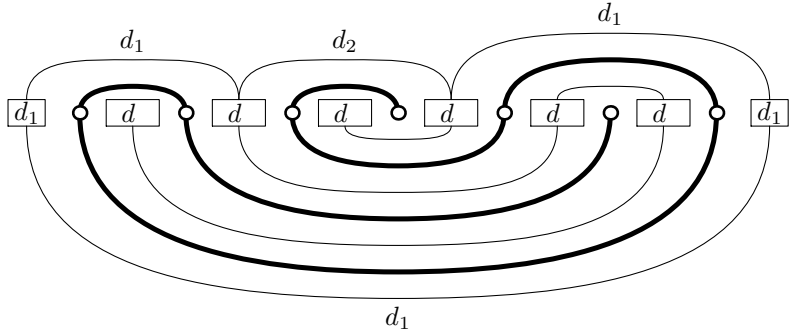}
 \caption{$M(2, 2, 3\,\,|\,\,7)$ with inserted cycle. \label{meanderfig}}
\end{center}
\end{figure}

Alternating following top and bottom edges, we see that the cycle $C$
has exactly $d$ vertices in every interval between vertices and
possibly $d_{1} \leq d$ vertices on each end beyond $u_{1}$ and beyond
$u_{n'}$.  Carefully counting, we see that the first part of our
partition has $a = a' + 2d_{1} + (a' - 1)d$ vertices.  Similarly, the
second part has $b = b' + 2d_{2} + (b' - 1)d$ and the third part has
$c = c' + 2d_{1} + (c' - 1)d$.  This means that $a + b = (a' + b')(d +
1)$ and $b + c = (b' + c')(d + 1)$ and these have a common factor of
$d + 1$, a contradiction.  This shows that $C$ must be empty so $G$ is
simply the path $P$.
\end{proof}

The following is an example to show that this argument does not work
when we break $n$ into more pieces.  Consider the meander $M = M(3, 2, 2, 2 \,\,|\,\, 9)$ pictured in Figure~\ref{meanderc-ex}.

\begin{figure}[ht!]
\begin{center}
 \epsfbox{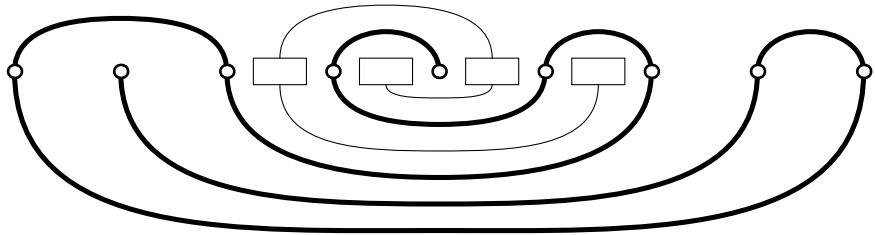}
 \caption{$M(3, 2, 2, 2\,\, |\,\, 9)$ with inserted cycle. \label{meanderc-ex}}
\end{center}
\end{figure}

Here we have broken the top into
$4$ pieces while leaving the bottom in one piece.  Notice that we can
add a cycle to this structure which does not pass through all the
intervals.  This happens because, as the number of pieces we have
increases, the number of dead ends also increases, allowing more
flexibility in the placement of the cycles.

The above illustrates the complexity of the meander graphs
$M(\x\,\,|\,\,\y)$ as the number of parts of $\x$ and $\y$ grow. At
the moment, the problem of classifying all Frobenius seaweed Lie
algebras seems to be out of reach. Of late, there has been a great
deal of interest in Frobenius Lie algebras. Perhaps these recent
developments will be instrumental in the development of a
classification theory.

\end{document}